\documentclass[12pt,reqno]{amsart}

 \usepackage{comment}
\usepackage{tikz}
\usepackage{latexsym}
\usepackage{graphicx}
\usepackage{amssymb}
 \graphicspath{{Figures/}}
 \usepackage{soul}

\usepackage{tikz}
\usepackage{caption}
\usepackage{subcaption}
\usetikzlibrary{
 decorations.pathmorphing,%
 decorations.markings,%
 decorations.pathreplacing,%
 decorations.text,%
 calc,%
 patterns,%
 shapes.geometric,%
 arrows,
 decorations.shapes,
 positioning,
 plotmarks,
 shadings,
 math,
 intersections
}

\renewcommand{\Re}{{\operatorname{Re}\,}}
\renewcommand{\Im}{{\operatorname{Im}\,}}

\newcommand{\sgn}{\operatorname{sgn}}

\renewcommand{\epsilon}{\varepsilon}

\newcommand{\I}{{\mathbf I}}

\newcommand{\R}{{\mathbb R}}

\newcommand{\e}{\epsilon}

\newcommand{\re}{\mathbb{R}}

\newcommand{\supp}{{\operatorname{supp\,}}}
\renewcommand{\phi}{\varphi}

\newcommand{\ep}{\varepsilon}

\newtheorem{theo}{Theorem}

\newtheorem{cor}{Corollary}[section]

\newtheorem{lem}[cor]{{ Lemma}}

\numberwithin{equation}{section}
\newcounter{remcounter}
\renewcommand{\theremcounter}{\arabic{remcounter}}

\newenvironment{rem}{%
\refstepcounter{remcounter}
\medskip

\noindent{\bf{Remark~\theremcounter.}}%
}%
{\\[.2em]}%
\renewcommand{\O}[1]{O_{#1}}

\usepackage{hyperref}

\definecolor{lGray}{RGB}{200, 200, 200}
\definecolor{yaizaColor}{RGB}{0, 153, 153}
\definecolor{certainty}{RGB}{64, 228, 198}
\definecolor{hope}{RGB}{228, 194, 64}

\definecolor{periodColor}{RGB}{255, 167, 105}
\definecolor{dark-green}{RGB}{135, 194, 130}
\tikzset{>=latex} 
\tikzset{font=\small}
\tikzset{mark size=1.5pt, mark options=thin}
\tikzset{pin distance=4pt,
 every pin edge/.style={<-, thin, shorten <= -2pt}}

\title[Lower bounds for Steklov Eigenfunctions]{Lower bounds for Steklov eigenfunctions}

\author{Jeffrey Galkowski}
\address{Department of Mathematics, University College London, London, United Kingdom}
\email{j.galkowski@ucl.ac.uk} 

\author{John A. Toth}
\address{Department of Mathematics and Statistics, McGill University, Montr\'eal, Canada}
\email{jtoth@math.mcgill.ca} 

\date{}

\begin{document}

\maketitle
\begin{abstract}  Let $(\Omega,g)$ be a compact, analytic Riemannian manifold with analytic boundary $\partial \Omega = M.$
We give $L^2$-lower bounds for Steklov eigenfunctions and their restrictions to interior hypersurfaces $H \subset \Omega^{\circ}$ in a geometrically defined neighborhood of $M$. Our results are optimal in the entire geometric neighborhood and complement the results on eigenfunction upper bounds in \cite{GaTo:19}. 
\end{abstract}

\section{Introduction} 
Let $(\Omega,g)$ be an $n+1$-dimensional, compact $C^{\infty}$ Riemannian manifold  with boundary $M$ and corresponding unit exterior normal $\nu$. By some abuse of notation, we also let $\nu$ denote a smooth vector field extension and $\gamma_{M}: C^0(\Omega) \to C^0(M)$ be the boundary restriction map.   Let ${\mathcal D}: C^{\infty}(M) \to C^{\infty}(M)$  be the associated Dirichlet-to-Neumann (DtN) operator defined by
\begin{equation} \label{DtN}
{\mathcal D} f :=  \gamma_{M} \partial_{\nu} u
\end{equation} 
where $u$ solves the Dirichlet problem
\begin{align} \label{dirichlet}
\Delta_{g} u(x) &= 0, \,\,\, x \in \Omega, &u(q) &= f(q), \,\,\, q \in M. \end{align}

The operator ${\mathcal D}$ is an ellptic, first order, self-adjoint pseudodifferential operator (see for example \cite[Section 7.11]{Tayl2}) with an $L^2$-normalized basis of eigenfunctions $\phi_{j}; j=1,2,....$ It is useful to work in the semiclasscial setting from the outset. Choosing $h^{-1} \in \text{Spec} \, {\mathcal D},$ the corresponding eigenfunction $\phi_h$ then satisfies the semiclassical eigenfunction equation
$$ h {\mathcal D}\phi_h = \phi_h.$$
The harmonic extension, $u_{h} \in C^{\infty}(\Omega),$ of a DtN eigenfunction $\phi_h$ is called a {\em Steklov eigenfunction.} Throughout the article, we will use the notation $u_h$ for these Steklov eigenfunctions which have $L^2$ normalized boundary restriction.

There has been a substantial amount of recent work devoted to the study of the asymptotic behaviour of the DtN eigenvalues and both DtN and Steklov eigenfunctions, including the asymptotics of eigenfunction nodal sets (see for example \cite{BLin,GP,GPPS,HL,PST,Sh,SWZ, Ze,Zh, Zhu} and references therein). 

For large eigenvalues, Steklov eigenfunctions possess both high oscillation inherited from the boundary DtN eigenfunctions and very sharp decay into the interior of $\Omega.$   As a consequence, even though Steklov eigenfunctions decay rapidly, the oscillation implies that the nodal sets have intricate structure. It has been conjectured \cite{GP} that the analogue of Yau's conjecture \cite{Yau,Yau2} for nodal  volumes holds in the Steklov case.  This was recently proved for real-analytic Riemann surfaces in \cite{PST}. In the case of smooth manifolds, the recent work~\cite{Decio:21a,Decio:21b} gives the best available polynomial upper and lower bounds on the nodal volume. Despite these bounds on the nodal volume, it is likely that a typical high energy Steklov eigenfunction exhibits regions of fixed sign with inner radius uniformly bounded from below~\cite{BrGa:20}.

  The question of decay of Steklov eigenfunctions into the interior of $M$ when $(M,g)$ is real analytic was first raised by Hislop--Lutzer \cite{HL} where they conjecture that the Steklov eigenfunctions decay into the interior as $e^{-d(x,\partial\Omega)/h}.$  In the special case where dim $\Omega =2$ and $\Omega$ is analytic, exponential decay with respect to $d(x,\partial\Omega)$ was proved in \cite{PST}, the case of general dimension and analytic $\Omega$ was handled in~\cite{GaTo:19} where the authors prove that 
  
  \begin{equation} \label{sup}
  \sup_{x \in \Omega_{\partial}(\ep_0)} |\partial_x^\alpha u_h(x)|e^{ d(x,\partial \Omega) - C_{sup}(\Omega) d^2(x,\partial \Omega) /h } \leq C_\alpha 
  \end{equation}
  where $\Omega_{\partial}(\ep_0)$ is a tubular neighbourhood of the boundary of width $\ep_0 = \ep_{0}(\Omega,M,g)>0$ and  $C_{\sup}$ is a constant depending on the second fundamental form of the boundary. Here, $\ep_0 >0$ is an $h$-independent positive constant that depends on the analyticity properties of the boundary and is difficult to quantify explicitly.

In this article, we consider the complementary question of lower bounds on Steklov eigenfunctions. As in the case of~\cite{GaTo:19}, we restrict our attention to the case of analytic $\Omega$ and $M$.  Our first result is  a {\em lower} bound for $L^2$ restrictions of eigenfunctions in a small $\ep_0$-neighbourhood of the boundary. In analogy with (\ref{sup}) we prove 

\begin{theo} \label{mainthm1}
 There exist  a neighbourhood, $\Omega_\partial(\ep_0)$ of $M = \partial \Omega$ and constants $C_j(\ep_0)>0; j=1,2$ such that for any connected component, $N$, of the boundary and any $\epsilon >0$ there are $C>0$ and $h_0>0$ such that for $h \in (0,h_0(\epsilon)]$ and $0\leq t\leq \e_0$, and $H_t:=\{x\,:\, d(x,N)=t\},$
  $$ e^{t+C_1(\ep_0) t^2/h} \| u_h \|_{L^2(H_t)} \geq Ce^{ - C_2(\ep_0) \, \epsilon/h}\|u_h\|_{L^2(N)}-C_1e^{-1/(hC_1)}\|u_h\|_{L^2(M)}.$$
\end{theo}

Here $\ep_0 = \ep_0(M,\Omega,g)>0$ is a possibly small constant (but independent of $h$)  that is the same in both the upper bounds (\ref{sup}) and lower bounds in Theorem \ref{mainthm1} and is difficult to quantify precisely.  

Our second result extends eigenfunction lower bounds to an  explicit {\em geometric} neighbourhood of the boundary.  Specifically, we use Carleman estimates to ``bootstrap" the local result in Theorem \ref{mainthm1} to the full geometric neighbourhood of the boundary.

To define the geometric tubular neighbourhood  more precisely, let $N$ be a connected component of $\partial \Omega$. We consider
the map $ \phi_N: N \times [0, r) \to \overline{\Omega}$ given by 
\begin{equation}
\label{e:phiDef} \phi_N (x, r) = \exp_{x} (- r \nu), \quad  r \in [0,r_0), \,\, x \in N,
\end{equation}
where $\exp$ is the Riemannian exponential map induced by the metric $g$ and $-\nu$ is the unit {\em interior} normal to $\partial \Omega$.
By the collar neighbourhood theorem, for sufficiently small $r_0>0.$ the map $\phi_N$  is a diffeomorphism onto its image $\phi_N( [0,r_0)).$ We let $r_{\max,N}$ be the maximal choice of $r_0$ with this property and set
\begin{equation} \label{geo}
\Omega_{N}(r_{\max,N}):= \phi ([0,r_{\max,N}) ). \end{equation}
We refer to $\Omega_{N}(r_{\max,N})$ as the {\em geometric} neighbourhood of the boundary component $N$. In the following, we sometimes abuse notation and just write $\Omega_{N}$ for $\Omega_{N}(r_{\max,N}).$ See Figure~\ref{f:domains} for a description of these domains for the annulus.

 \begin{theo} \label{mainthm2}
 Let $\Omega$ be an analytic manifold with analytic boundary, $M=\partial\Omega$, and  $\Omega_{N} \subset \overline{\Omega}$  (as in~\eqref{geo}) be the g Fermi neighbourhood of the connected component, $N$, of $M$ and $0<t<r_{\max,N}$. Then, for any tubular neighbourhood, $U_{H_t}$, of $H_t:=\varphi_N(N,t)$ and $\e>0$, there are $h_0>0$ and $C>0$ such that for $h \in (0,h_0],$
 $$ \| e^{ \psi_N(t) /h} \, u_h \|_{L^2(U_{H_t})} \geq C e^{-\e/h}(\|\,u_h\|_{L^2(N)}-e^{-1/Ch}\|u_h\|_{L^2(M)}),$$
 where 
 $$
 \psi_N(x_{n+1})=\int_0^{x_{n+1}} e^{\int_0^s Q_{\sup,N}(t)dt}ds,\qquad Q_{\sup,N}(t):=\sup \{ Q(t,x',\xi')\,:\, x'\in N,\, |\xi'|_{g_t(x')}=1\},
 $$
 $g_t$ is the metric induced on $H_t$, and $Q(t,x',\xi')$ is the second fundamental form on $H_t$ induced by the inward pointing normal.
 \end{theo}
 \noindent The examples of the disk, cylinder, and annulus in Sections~\ref{s:e1}-\ref{s:e2} show that Theorem~\ref{mainthm2} is optimal.

 \begin{rem}
  Notice that, although the right hand sides of the estimates in Theorems~\ref{mainthm1} and~\ref{mainthm2} have an error $e^{-1/Ch}$ with a constant $C$ depending on the analyticity properties of $\Omega$ and $M$, these do not cause losses in the estimates when $\|u_h\|_{L^2(N)}\gg e^{-1/Ch}\|u_h|_{M}\|_{L^2}$. Since there are finitely many boundary components, there are \emph{always} boundary components where this is the case.  
  
  Furthermore, we may replace $N$ in Theorem~\ref{mainthm2} by a union of boundary components, $\tilde{N}:=\cup_{j=1}^L N_j$  by applying Theorem~\ref{mainthm2} for each $N_j$ if we replace $\psi_N$ by $\psi_{\tilde{N}}$. 
 \end{rem}

By Taylor expansion at the boundary $M = \{ x_{n+1} = 0 \},$
  \begin{equation} \label{taylor}
  \psi_N(x_{n+1}) = x_{n+1}  + \frac{Q_{\sup,N}(0)}{2} x_{n+1}^2 + O(x_{n+1}^3),
  \end{equation}
  where $Q_{\sup,N} (0)$ is the maximum of the second fundamental form along $N\subset \partial\Omega$. Thus, near the boundary, eigenfunction decay is given to first order by $x_{n+1} = d(x, M)$. However, when the boundary is strictly convex, the quadratic correction in (\ref{taylor}) is actually {\em negative} and so the rate of decay in our estimate may be faster than $e^{-d(x,M)/h}$. The simple example of the disc (see section \ref{s:e1}) shows that this extra decay \emph{does} occur. Likewise, when a boundary component is strictly concave, the quadratic correction is positive, producing a sub-linear rate of decay. This behavior can be seen in the example of the annulus (see section~\ref{s:e2}).


\begin{figure}

\centering
\begin{subfigure}[b]{.45\textwidth}
\centering
\includegraphics[width=\textwidth]{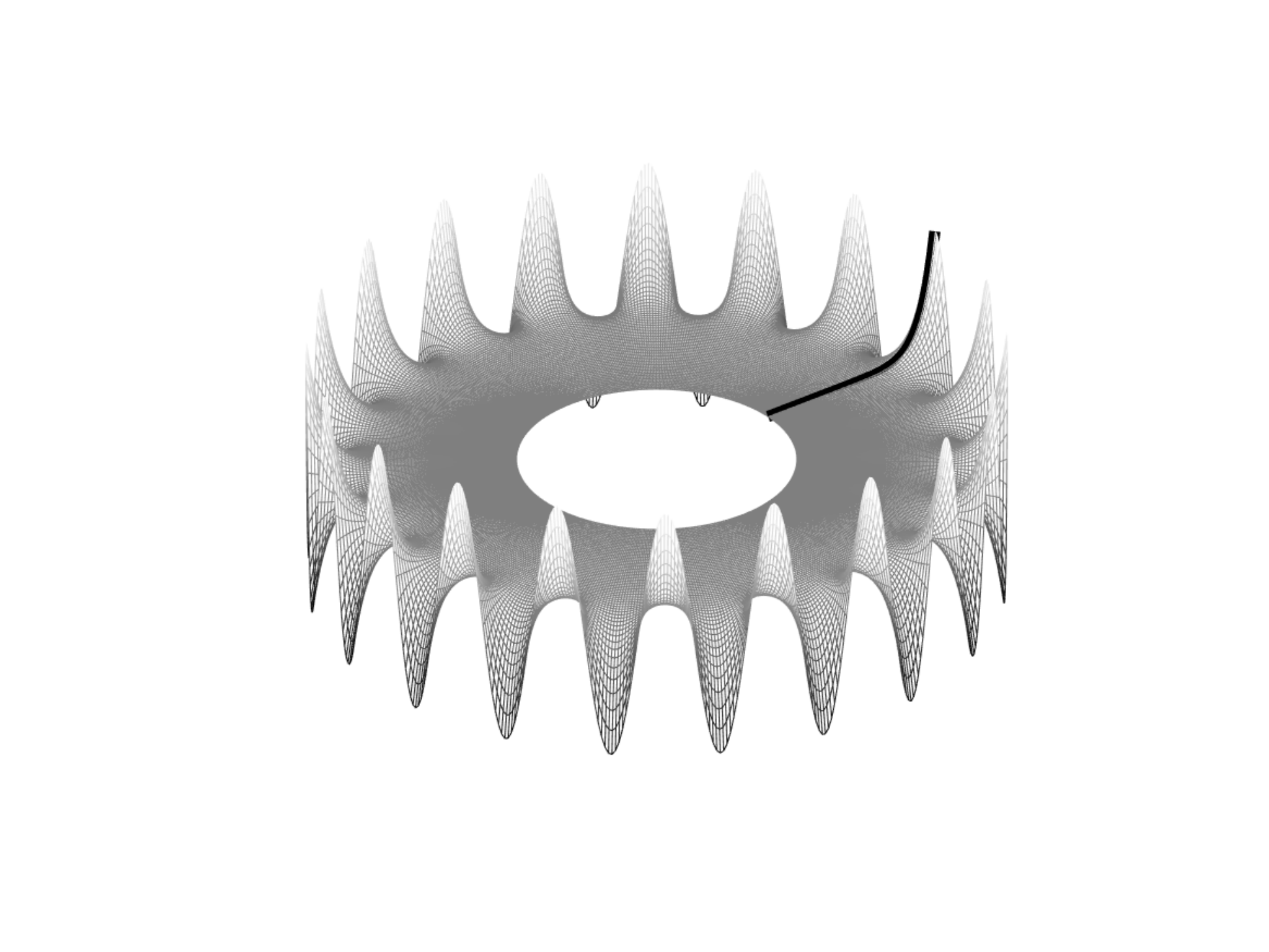}
\caption{$u_{\sigma_{20,1}}$}
\end{subfigure}
\begin{subfigure}[b]{.45\textwidth}
\centering
\includegraphics[width=\textwidth]{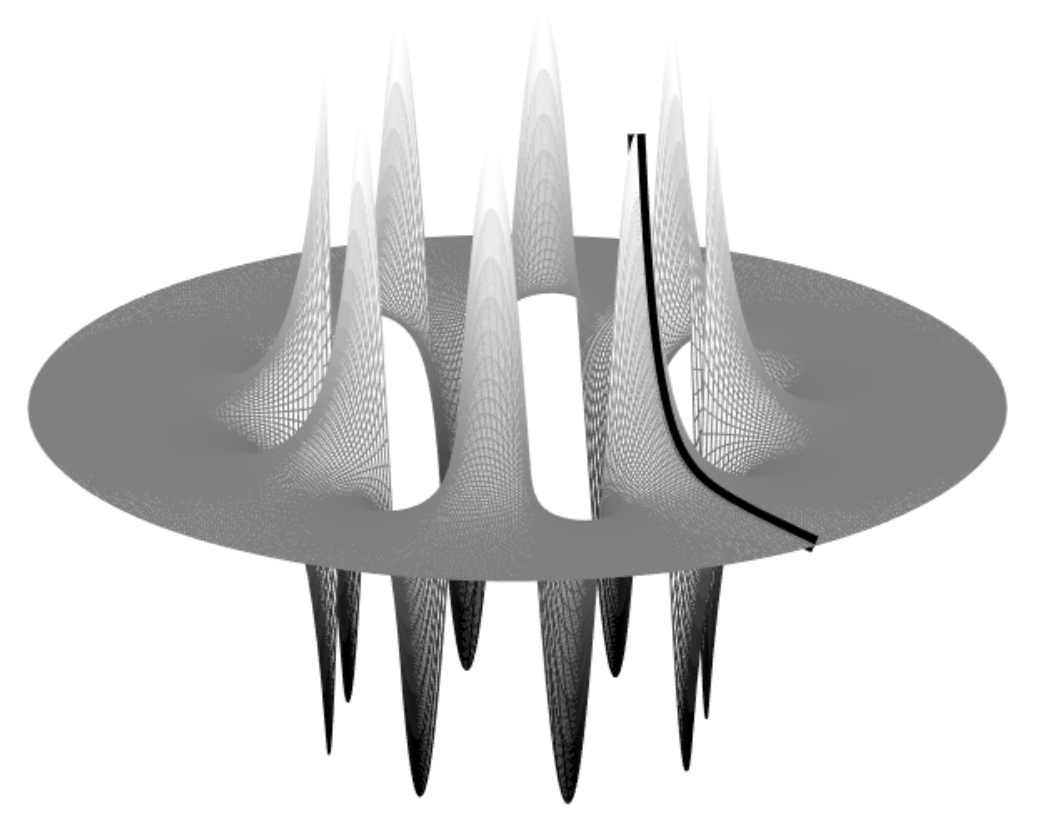}
\caption{$u_{\sigma_{8,2}}$}
\end{subfigure}
\caption{\label{f:annulusEigenfunctions}Steklov eigenfunctions with $\sigma\sim20$ on the annulus with $r_0=.4$. The black line shows the lower bound predicted by Theorem~\ref{mainthm2}. The labels $\sigma_{k,j}$ are as in~\eqref{e:sigmaKJ}.}
\end{figure}

Our final results concerns lower bounds of the $L^2$-restriction of eigenfunciton Cauchy data along $H \subset \Omega_{\partial}(r_{\max,\partial\Omega}).$ We recall that given a hypersurface $H \subset \Omega,$ the {\em Cauchy data} along $H$ is the pair
$$CD_H(h):= (u_h|_H, h\partial_{\nu} u_h |_H).$$

The lower bound in Theorem \ref{mainthm2} combined with a potential layer formula in the tube $U_H(\ep)$ allows us prove goodness for $CD_H(h)$ for hypersurfaces $H$ potentially far inside $\Omega$. 


\begin{theo} \label{mainthm3}
 Let $\Omega$ be an analytic manifold with analytic boundary, $M=\partial\Omega$ and $H_t:=\varphi_{\partial\Omega}(\partial\Omega,t)$. 
Then for $0<t<r_{\max,\partial\Omega}$ and $\epsilon >0$ there are $C(\ep)>0$ and $h_0(\ep)>0$ such that for $h \in (0,h_0(\epsilon)]$, 
$$ e^{\psi_{\partial\Omega}(t)/h} \, \big( \| u_h \|_{L^2(H_t)} + \| h \partial_{\nu} u_h \|_{L^2(H_t)} \big)  \geq C(\ep) e^{- \epsilon/h}\|u_h\|_{L^2(M)}.$$
\end{theo}\
Theorems~\ref{mainthm2} and~\ref{mainthm3} are optimal in a sense made precise in the next section.

\subsection{Organization of the paper}

The proof of Theorem \ref{mainthm1} follows first from the existence of a parametrix for the Poisson kernel modulo analytic errors, and second, from the construction of an approximate inverse for this operator valid at frequencies $\lesssim h^{-1}$.  However, the parametrix construction is only valid in a collar of radius $\ep_0 = \ep_0(M,\Omega,g)$ that while $h$-independent, is possibly smaller than $r_{\max,N}.$ The proof of Theorem~\ref{mainthm1} is taken up in section \ref{local}

The proof of  Theorem \ref{mainthm2} is given in section \ref{s:carleman}.  Here, we use the local result in Theorem \ref{mainthm1} as a control estimate and  use Carleman estimates to extend the lower bound in $\Omega_{\partial}(\ep_0)$ to the full geometric neighbourhood $\Omega_{N}(r_{\max,N})$ of the boundary. Finally Section~\ref{upper} applies a layer potential formula together with Theorem~\ref{mainthm2} to prove Theorem~\ref{mainthm3}.

\bigskip

\noindent\textsc{Acknowledgements:} J.G. is grateful to the EPSRC for support under Early Career Fellowhip EP/V001760/1. J.T. was
partially supported by NSERC Discovery Grant \# OGP0170280 and by the French National Research Agency project Gerasic-ANR- 13-BS01-0007-0.


\section{Examples}
\label{s:ex}
\subsection{The Disk}
\label{s:e1}
Let $\Omega=B(0,R)\subset \re^2$. Then the Steklov eigenvalues are precisely $\sigma=0,\frac{1}{R},\frac{2}{R}\dots $ with corresponding Steklov eigenfunctions given by 
\begin{equation}
\label{e:e1}u^{\pm}_{k}=\frac{1}{\sqrt{2\pi R}R^k}r^{k}e^{\pm i{k}\theta},\quad \sigma =\frac{{k}}{R}.
\end{equation}
In particular, letting $h=\sigma^{-1}={k}^{-1}R$, 
$$u^{\pm}_{k}=\frac{1}{\sqrt{2\pi R}}e^{[R\log (1-(R-r)/R))]/h}e^{i\theta R/h}.$$

Now, we recall that the metric in Fermi normal coordinates relative to $\partial B(0,R)$ (i.e. with $x_{2}=R-r$) is given by
$$
\xi_{2}^2+\frac{1}{(R-x_{2})^2}\xi_\theta^2,
$$
and hence, the metric induced on $H_{t}$ and second fundamental form on $H_t$ are given by 
$$
|\xi_\theta|_{g_t}^2= \frac{\xi_\theta^2}{(R-t)^2},\qquad Q(t,\theta,\xi_\theta)=\frac{\xi_\theta^2}{(R-t)^3}.
$$
In particular, 
$$
Q_{\sup}(t)= \frac{1}{R-t},
$$
and therefore, 
$$
\psi(x_{2})=\int_0^{x_{2}}e^{\int_0^sQ_{\sup}(t)dt}ds= R\log(1-\tfrac{x_{2}}{R}).
$$

Undoing, the change of variables and applying Theorem~\ref{mainthm3} to obtain a lower bound, we have that
$$
C(\e)e^{-\e/h}\leq e^{R\log(1-\tfrac{R-r_0}{R})/h}[\|u^{\pm}_k\|_{L^2(r=r_0))}+\|h\partial_r u^{\pm}_k\|_{L^2(r=r_0)}]\leq C.
$$
In particular, the exponential weight  in Theorems~\ref{mainthm2} and~\ref{mainthm3} is optimal.

The case of spheres in higher dimensions is nearly identical if we replace $e^{\pm ik\theta}$ by a spherical harmonic.

\subsection{Cylinders}
Let $(M,g)$ be a real analytic manifold of dimension $n$ without boundary and $\Omega=(-1,1)_s\times M_x$ with metric $ds^2+g(x).$ Then 
$$\Delta_\Omega=\partial_s^2+\Delta_M.$$
Let $\varphi_{k}$ be an orthonormal basis for $L^2(M)$ with 
$$(-\Delta_M-\lambda_{k}^2)\varphi_{k}=0.$$
Then the Steklov eigenfunctions are given by 
$$u_h(s,x)=\frac{\cosh(\lambda_{k}s)}{\cosh(\lambda_{k})}\varphi_{k}(x),\qquad v_h(x,s)=\frac{\sinh(\lambda_{k}t)}{\sinh(\lambda_{k})}\varphi_{k}(x)$$
with Steklov eigenvalues $\sigma_{k}=\lambda_{k}\tanh(\lambda_{k})$ and $\sigma'_{k}=\lambda_{k}\coth(\lambda_{k})$ respectively. Notice that  
$$\cosh(x)=\frac{1}{2}e^{|x|}+O(e^{-|x|}),\qquad \sinh(x)=\frac{\sgn(x)}{2}e^{|x|}+O(e^{-|x|}),$$
and
$$\sigma_{k}=\lambda_{k}(1+O(e^{-\lambda_{k}})).$$

It is easy to see that $Q_{\sup}(s)\equiv 0$ and hence, taking $N_L=\{s=-1\}$, we have $\psi_{N_L}(s)=1+s$, and taking $N_R=\{s=1\}$, we have $\psi_{N_R}=1-s$. Combining the lower bounds from Theorem~\ref{mainthm2} applied with $N_L$ and $N_R$, we obtain optimal lower bounds on $M$. Similarly, we obtain optimal lower bounds with an application of Theorem~\ref{mainthm3}, but this time the hypersurface is given by $H_t=\{s=-1+t\}\sqcup \{s=1-t\}$ and the Theorem is valid for $0<t<1$. 

\begin{rem}
Notice that a cylinder has the unusual feature that there are Steklov eigenfunctions with non-negligible mass on multiple boundary components. This is why one must apply Theorem~\ref{mainthm2} twice (once from the left hypersurface and once from the right) to obtain the correct lower bounds.
\end{rem}

\subsection{The Annulus}
\label{s:e2}
Now, consider $B(0,1)\setminus B(0,r_0)\subset \re^2$. Then a simple computation shows that the Steklov eigenvalues are the roots of 
$$p_{k}(\sigma)=\sigma^2-\sigma {k}\left(\frac{1+r_0}{r_0}\right)\left(\frac{1+r_0^{2{k}}}{1-r_0^{2{k}}}\right)+\frac{{k}^2}{r_0},\quad {k}=0,1,\dots$$
with corresponding eigenfunctions
\begin{equation}
\label{e:e2}
u^{\pm}_\sigma(r,\theta)=C_{{k},\sigma}e^{\pm ik\theta}\left(r^{k}+\frac{{k}-\sigma}{{k}+\sigma}r^{-{k}}\right).
\end{equation}
See Figure~\ref{f:annulusEigenfunctions} for graphs of two such eigenfunctions.

It is easy to show that the roots of $p_n(\sigma)$ have 
\begin{equation}
\sigma_{{k},1}={k}+O({k}r_0^{2{k}}),\quad \sigma_{{k},2}=\frac{{k}}{r_0}+O({k}r_0^{2{k}}).
\label{e:sigmaKJ}
\end{equation}
Then, 
\begin{gather*} 
u^{\pm}_{\sigma_{{k},1}}=\frac{1}{\sqrt{2\pi}}e^{\pm i{k}\theta}(r^{k}+O(r_0^{2k})r^{-k}),\quad u^{\pm}_{\sigma_{{k},2}}=\frac{1}{\sqrt{2\pi r_0}}r_0^{k}e^{\pm i{k}\theta}(r^{-k}+O(1)r^k).
\end{gather*}
The case of $u_{\sigma_{{k},1}}$ is identical to that for the disk when $r>r_0+\e$, so we focus on $u_{\sigma_{{k},2}}$. Let $h=\sigma_{{k},2}^{-1}=r_0{k}^{-1}+O(e^{-c{k}}).$ Then, for $r<1-\e$,
$$|u^{\pm}_{\sigma_{{k},2}}(r,\theta)|\geq\frac{1}{\sqrt{2\pi r_0}}e^{-r_0\log[1+ (r-r_0)/r_0]/h}(1+O(e^{-c/h})).$$
Using exactly the same computation as for the interior of the disk, one again sees that the lower bound in Theorem~\ref{mainthm2} is optimal for $u_{\sigma_{k,2}}$. Indeed, $\|u_{\sigma_{k,2}}\|_{L^2(\partial B(0,r_0))}=1+e^{-c/h}$ and $\psi(r)=r_0\log(1+(r-r_0)/r_0).$

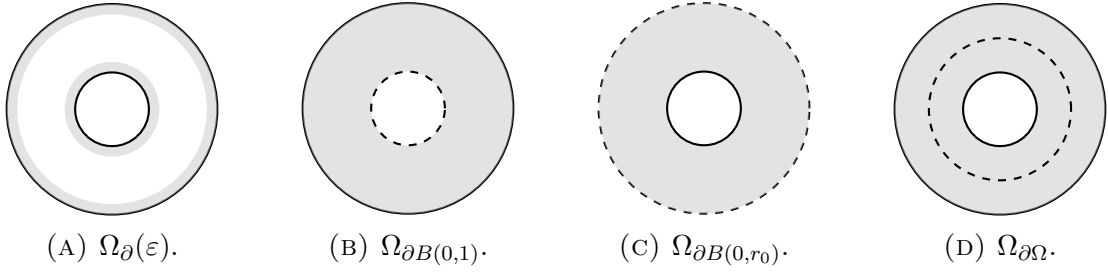
\begin{figure}
\centering
\begin{subfigure}[b]{.23\textwidth}
\centering
\begin{tikzpicture}
\begin{scope}[scale=.7]
\draw[thick] (0,0) circle (2);
\fill[fill=lGray,opacity=.5] (0,0) circle (2);
\fill[fill=white](0,0) circle(1.8);
\fill[fill=lGray,opacity=.5] (0,0) circle (.9);
\draw[fill=white,thick](0,0)circle (.7);
\end{scope}
\end{tikzpicture}
\caption{$\Omega_{\partial}(\e)$.}
\end{subfigure}
\begin{subfigure}[b]{.23\textwidth}
\centering
\begin{tikzpicture}
\begin{scope}[scale=.7]
\draw[thick] (0,0) circle (2);
\fill[fill=lGray,opacity=.5] (0,0) circle (2);
\draw[fill=white,thick,dashed](0,0)circle (.7);
\end{scope}
\end{tikzpicture}
\caption{$\Omega_{\partial B(0,1)}$.}
\end{subfigure}
\begin{subfigure}[b]{.23\textwidth}
\centering
\begin{tikzpicture}
\begin{scope}[scale=.7]
\draw[thick,dashed] (0,0) circle (2);
\fill[fill=lGray,opacity=.5] (0,0) circle (2);
\draw[fill=white,thick](0,0)circle (.7);
\end{scope}
\end{tikzpicture}
\caption{$\Omega_{\partial B(0,r_0)}$.}
\end{subfigure}
\begin{subfigure}[b]{.23\textwidth}
\centering
\begin{tikzpicture}
\begin{scope}[scale=.7]
\draw[thick] (0,0) circle (2);
\fill[fill=lGray,opacity=.5] (0,0) circle (2);
\draw[fill=white,thick](0,0)circle (.7);
\draw[thick,dashed](0,0) circle (1.35);
\end{scope}
\end{tikzpicture}
\caption{$\Omega_{\partial \Omega}$.}
\end{subfigure}
\caption{\label{f:domains}Regions of applicability for Theorems~\ref{mainthm1},~\ref{mainthm2}, and~\ref{mainthm3} for $\Omega=B(0,1)\setminus \overline{B(0,r_0)}$. The relevant regions are shaded in gray with dashed lines not included in the region.}
\end{figure}

\section{Lower bounds sufficiently close to the boundary} \label{local}

 

The main goal of this section is the proof of Theorem \ref{mainthm1}. As we already indicated in the introduction, here $\ep_0 = \ep_0(M,\Omega,g)>0$ is a possibly small constant that depends on the analyticity properties of $M.$ As such, it is difficult to quantify. 

\subsection{Notation for pseudodifferential operators}
Below, we will need notation for semiclassical pseudodifferential operators. We say that \emph{$a\in C^\infty(T^*M)$ is a symbol of order $k$}, and write $a\in S^k(T^*M)$, if 
$$
|\partial_x^\alpha\partial_\xi^\beta a(x,\xi)|\leq C_{\alpha\beta}\langle \xi\rangle^{k-|\beta|},\qquad \langle \xi\rangle =(1+|\xi|^2)^{1/2}.
$$
Note that, $a$ may implicitly depend on the small parameter $h$. We also define the set of semiclassical pseudodifferential operators of order $k$, $\Psi^k(M)$ as in~\cite[Chapter 14]{Zw} or~\cite[Appendix E]{DyZw:19}. Note, in particular, that semiclassical pseudodifferential operators of order $-\infty$ are smoothing, but their norms do not vanish as $h\to 0$.  We also define the elliptic set of a pseudodifferential operator as in~\cite[Definition E.31]{DyZw:19}.

\subsection{Analytic symbols}

In this section, we will need the notion of a classical analytic symbol, which we recall from~\cite{Sj:82} (see also~\cite{SU}). We say that $a$ is \emph{classical analytic of order $k$} and write $a\in S^k_{cl,a}$ if there exist $C_0>0$ and functions $a_j$ analytic on a fixed neigbhorhood of $T^*M\setminus \{0\}$, homogeneous degree $j$, satisfying
$$
\Big|a_j\Big(x,\frac{\xi}{|\xi|}\Big)\Big|\leq C_0^{j+1}(j+1)!,
$$
and for every $C_1>0$ large enough, there is $C_2>0$ such that 
$$
\Big| a(x,\xi)-\sum_{0\leq j\leq |\xi|/C_1}a_{k-j}(x,\xi)\Big|\leq C_2e^{-|\xi|/C_2},\qquad |\xi|\geq 1.
$$
The key fact that we will use about such symbols is that, after rescaling $\xi\to \xi/h$, it is possible to deform contours away from $|\xi|=0$ modulo errors of order $e^{-1/Ch}$. 

We also recall the notion of a semiclassical, classical analytic symbol. We say that $a\in S^{k}_{h,cl,a}$ provided there are $a_j\in S^{k-j}(T^*M)$, $h$ independent and analytic in a conic neighborhood of $T^*M$, and $C_0$, $C>0$ such that 
$$
\begin{gathered}\partial_{x}^{l_1}\partial_{\xi}^{l_2}\bar{\partial}_{(x,\xi)}a=O_{l_1,l_2}(e^{-\langle \xi\rangle /Ch}),\\
\Big|a-\sum_{0\leq j\leq |\xi|/C}h^ja_j(x,\xi)\Big|\leq Ce^{-\langle \xi\rangle/Ch},\qquad |a_j|\leq C_0 C^jj!\langle \xi\rangle^{k-j}.
\end{gathered}
$$
Contours can, again modulo errors of the form $e^{-1/Ch}$, be readily deformed when this type of symbol is involved. 

\subsection{A geometric FBI transform}
We also recall a particular Fourier-Bros-Iagolnltzer (FBI) transform on $M$. Define the operator $T:C^\infty(M)\to C^\infty(M)$ by 
\begin{equation}
\label{e:FBI}
Tu(x,\xi)=\frac{1}{(2\pi h)^{3n/4}}\int e^{i(\langle \exp_{y}^{-1}(x),\xi\rangle +i\frac{\langle \xi\rangle}{2} d(x,y)^2)/h}a(x,\xi,y)u(y)dy,
\end{equation}
where $a\in S_{h,cla}^{n/4}$, which is uniformly bounded from $L^2(M)\to L^2(T^*M)$, and has a left inverse $S:L^2(T^*M)\to L^2(M)$ given by 
\begin{equation}
\label{e:FBI2}
Sv(x)=\frac{1}{(2\pi h)^{3n/4}}\int e^{-i(\langle \exp_{x}^{-1}(y),\xi\rangle -i\frac{\langle \xi\rangle}{2} d(x,y)^2)/h}b(x,\xi,y)v(x,\xi)dxd\xi,
\end{equation}
for some $b\in S_{h,cla}^{n/4}$, which is also uniformly bounded. As in~\cite[Theorem 2]{GaTo:19}, the FBI transform and its left inverse will be useful when localizing modulo exponential errors.

\subsection{Preliminaries on the Poisson Operator}

Let $P: C^{\infty}(\partial \Omega) \to C^{\omega}(\Omega)$  be a parametrix for the Poisson operator modulo analytic errors of the form
\begin{equation}
\label{poissonkernel}
[Pf](x)=\frac{1}{(2\pi )^{n}}\int e^{i(\psi(x_{n+1},x',\xi')-\langle y',\xi'\rangle)}a(x_{n+1},x',\xi')f(y')dy'd\xi'.
\end{equation}
that is, there is $K_a$ an operator with analytic kernel such that 
$$
\Delta_g (P+K_a)=0 \text{ in }\Omega\qquad (P+K_a)|_{\partial \Omega}=f. 
$$
Such an operator exists by~\cite{SU}  (see also~\cite{Le:18,St:14,Ze:12,Gu:97}). In addition, $a$ is classical analytic of order $0$ and $\psi\in C^\infty([0,\epsilon)\times T^*M\setminus\{0\}) $ satisfies, 
\begin{equation}
\label{e:poissonPhase}
\partial_{x_{n+1}}\psi-i\sqrt{r(x,\partial_{x'}\psi)}=0,\qquad \psi(0,x',\xi')=\langle x',\xi'\rangle,
\end{equation}
where, in Fermi normal coordinates, the symbol of $-\Delta_g$ is $\xi_n^2+r(x,\xi')$.

Let $t>0$ and define the smooth hypersurface
$$
H_t:=\{(x',t)\mid x'\in \partial \Omega\}
$$
 In the following, we identify $C^{\infty}(\partial \Omega)$ with $C^{\infty}(H_t)$ under the diffeomorphism $\partial \Omega \ni x' \mapsto (x',t) \in H_t.$

Let $\varphi\in C_c^\infty(0,\infty)$ with $\varphi(x)\equiv 1$ near {$\{x\equiv 1\}$} consider the family of operators   $E_h :L^2(\partial \Omega) \to L^2(\partial \Omega)$ given by
\begin{equation}
\label{e:defInverse}
[E_hu](x')=\frac{1}{(2\pi h)^{n}}\int e^{i(\langle x',\xi'\rangle -\psi(t,y',\xi'))/h}\varphi(|\xi'|_{g(x')})u(y')dy'd\xi'.
\end{equation}
Let also $\gamma_H$ denote the restriction operator from $\Omega$ to $\{H_t\}$.
\begin{rem}
Note that, because we include the compactly supported cutoff $\varphi$ in the amplitude of the integral defining $E_h$, $E_h$ is well-defined as an operator on $L^2(\partial\Omega)$. It would, however, be possible to define a left inverse for $P$ acting on sufficiently analytic functions as in e.g.~\cite{Gu:97}, but this is not necessary here.
\end{rem}

\begin{lem}
\label{l:itsAnInverse}
Let $T$ and $S$ be the FBI transform and its left inverse from~\eqref{e:FBI} and~\eqref{e:FBI2}.There is $\ep_0>0$ such that for $0<t<\ep_0$,  $A_H:=E_h \gamma_H PS\varphi(|\xi'|_g)T\in \Psi^{-\infty}(\partial\Omega)$ .    Moreover, $A$ is elliptic on   
$$
\{ (x',\xi')\,:\,|\xi'|_{g(0,x')}=1\}.
$$ 
\end{lem}
\begin{proof}
We start by computing the kernel of $E_h\gamma_HPS\varphi(|\eta'|_g):$
\begin{align*}
&[E_h \gamma_HPS\varphi(|\eta'|_g)] (x',y',\eta')\\
&=\frac{1}{(2\pi h)^{2n+3n/4}}\int e^{i(\langle x',\xi'\rangle -\psi(t,z',\xi')+\psi(t,z',\omega')-\langle s',\omega'\rangle-\langle \exp_{s'}^{-1}(y'),\eta'\rangle +\frac{i{\langle \eta'\rangle}}{2}d(y',s')^2)/h}\\
&\qquad\qquad a(t,z',\omega'/h)\varphi(|\xi'|_{g(x')})b(y',\eta',s')\varphi(|\eta'|_{g(y')})ds' d\omega'dz'd\xi'.
\end{align*}
We start by formally computing the critical points of the phase in $z',\omega',s' $. Let
$$
\Phi=\langle x',\xi'\rangle-\psi(t,z',\xi')+\psi(t,z',\omega')-\langle s',\omega'\rangle-\langle \exp_{s'}^{-1}(y'),\eta'\rangle +\frac{i{\langle \eta'\rangle}}{2}d(y',s')^2.
$$
Then, using Fermi normal coordinates centered around $y'$ in the $s'$-variables to compute
$$
\begin{gathered}
\partial_{z'}\Phi= \partial_{z'}(\psi(t,z',\omega')-\psi(t,z',\xi')),\qquad \partial_{\omega'}\Phi=\partial_{\omega'}\psi(t,z',\omega')-s',\\
\partial_{s'}\Phi=\eta'-\omega'+i\langle \xi'\rangle(s'-y').
\end{gathered}
$$
Now, observe that~\eqref{e:poissonPhase} implies that
\begin{equation}
\label{e:hessian}
\partial^2_{(z',\omega',s')}\Phi=\begin{pmatrix} 0&I&0\\I&0&-I\\0&-I&i{\langle\eta'\rangle}\end{pmatrix} + \begin{pmatrix} O(t)&O(t)&0\\O(t)&O(t)&0\\0&0&0\end{pmatrix} 
\end{equation}
and hence that the phase is non-degenerate for $t$ small. Moreover, by Taylor expansion,
$$
 \partial_{z'}\Phi = (I+ t  \, A(t,z',\omega',\xi'))( \omega'-\xi'),
$$
for some $A\in C^\infty$. In particular, denoting the critical points by $(\omega'_c,z'_c, s'_c)$, we have $\omega'_c=\xi'$ for $t$ small enough and hence $s'_c=y'+i{\langle \eta'\rangle^{-1}}(\eta'-\xi')$,  and $z'_c=s_c'+ O(t) $. Thus, 
\begin{align*}
\Phi(x,\xi',z'_c,\omega'_c,s'_c,y',\eta')&=\langle x',\xi'\rangle +\langle s'_c,\eta'-\xi'\rangle -\langle y',\eta'\rangle +\frac{i{\langle \eta'\rangle}}{2}d(y',s_c')^2\\
&=\langle x'-y',\xi'\rangle +\frac{i}{2{\langle\eta'\rangle}}(\xi'-\eta')^2.
\end{align*}
. 

{We will need to deform the contour to a good contour in order to perform complex stationary phase~\cite[Theorem 2.8]{Sj:82}. However, before doing this, we must show that the region near $\omega'=0$ can be neglected. Since the integrand is supported on $|\eta'|>c>0$, this can be done by deforming the contour in $s'$ alone by $s'\mapsto s'+i\delta \langle \eta'\rangle^{-1}(\omega'-\eta')$.}

{Next, we need to find a good contour for the phase. That is, we want to find a contour, $\Gamma$, on which the critical point, $\rho_c=(\omega'_c,z'_c,s'_c)$ lies and $\Im \Phi|_{\Gamma}(\rho)\geq c|\rho-\rho_c|^2$. To do this, we use the Hessian~\eqref{e:hessian} to choose a contour such that 
$$
\Im \langle \partial^2_{(z',\omega',s')}\Phi|_{\rho_c}\, v ,v\rangle \geq c|v|^2,\qquad v=\rho-\rho_c, \qquad \rho\in \Gamma.
$$}

For this, let $\chi \in C_c^\infty(\mathbb{R};[0,1])$ with $\chi \equiv 1$ near 0 and deform the contour to $\Gamma_1$, with $\Gamma_r$ defined for $r\in [0,1]$ by
\begin{multline*}
\Gamma_r:(z',\omega',s')\mapsto \Big( z'+r(z_c+\frac{i\delta \omega'}{\langle \omega'\rangle}),   \,\, \omega' +\xi'+i r\delta z',  \,\, s' +y'+ir{\langle\eta'\rangle^{-1}}(\eta'-\xi')\chi(\delta^{-1}{\langle\eta'\rangle^{-1}}|\eta'-\xi'|)\\+ir\delta(\frac{\eta'-\xi'}{|\eta'-\xi'|}(1-\chi(\delta^{-1}{\langle\eta'\rangle^{-1}}|\eta'-\xi'|))\Big).
\end{multline*}
First, note that for $|\omega'|\gg 1$, $r\in[0,1]$,  $\Im \Phi|_{\Gamma_r}\geq c|\omega'|$. Thus, the terms coming from infinity in $\omega'$ vanish and the contour deformation from $\Gamma_0=\mathbb{R}^{3n}$ to $\Gamma_1$ is justified. 

Moreover, for $|\xi'-\eta'|\leq\delta$, the phase satisfies 
\begin{multline*}
\Phi|_{\Gamma_1}= \langle x'-y',\xi'\rangle  +\frac{i}{2{\langle \eta'\rangle}}(\xi'-\eta')^2+\langle z'+\frac{i\delta \omega'}{\langle \omega\rangle},  \omega' +i \delta z'\rangle +\langle s', \omega'+i\delta z\rangle
+i\frac{{\langle \eta'\rangle}|s'|^2}{2}\\+ O(t(|\omega' +i \delta z'|^2+|z'+\frac{i\delta \omega'}{\langle \omega'\rangle}|^2)),
\end{multline*}
and for $|\xi'-\eta'|\geq \delta$, $\Im \Phi|_{\Gamma_1}\geq c\delta$. In particular, for $t\ll \delta$, $\Gamma_1$ is a good contour for $\Phi$ and we may apply the method of analytic stationary phase in $(z',\omega', w')$ to obtain
\begin{align*}
&[E_h \gamma_HPS\varphi(|\eta'|_g)] (x',y',\eta')\\
&=\frac{1}{(2\pi h)^{n+\frac{n}{4}}}\int e^{i[\langle x'-y',\xi'\rangle+\frac{i}{2{\langle\eta'\rangle}}(\xi'-\eta')^2]/h}(a(t,z'_c,\xi'/h)\varphi(|\xi'|_{g(x')})\varphi(|\eta'|_{g(x')})\bar{b}(y',\eta',w'_c) \\
&\qquad +O(h)_{C_c^\infty} )d\xi' +O(e^{-C/h}).\end{align*}
Here, we crucially use that $a(t,z',\omega'/h)$ is uniformly bounded with all derivatives when $|\omega|'>c>0$.

Finally, we precompose with $T$ to obtain the phase
$$
\tilde{\Phi}= \langle x'-w',\xi'\rangle+\frac{i}{2{\langle\eta'\rangle}}(\xi'-\eta')^2 +\frac{i{\langle\eta'\rangle}}{2}(w'-y')^2+\langle w'-y',\eta'\rangle
$$ and, using that $\mathbb{R}_{(w',\eta')}^{2n}$ is a good contour, we may perform (analytic) stationary phase in $w',\eta'$ to obtain
$$
[E_h \gamma_HPS\varphi(|\xi|_g)T] (x',y')
=\frac{1}{(2\pi h)^{n}}\int e^{i[\langle x'-y',\xi'\rangle]/h} \, \tilde{a}(x',y',\xi)  \, d\xi' +O(h^\infty)_{C^\infty},
$$
where $\tilde{a}\in C_c^\infty$ and $|\tilde{a}(x',x',\xi')|>c>0$ on $|\xi'|_g=1$. In particular, $E_h\gamma_HPS\varphi(|\xi'|_g)T$ is a semiclassical pseudodifferential operator of order zero which is elliptic as claimed. 
\end{proof}

\subsection{Lower bounds in a fixed non-geometric neighborhood: Proof of Theorem \ref{mainthm1}}

\begin{proof}
Recall from~\cite[Theorem 2]{GaTo:19} that, with 
$T$ and $S$ as in~\eqref{e:FBI} and~\eqref{e:FBI2} respectively, for $\varphi\in C_c^\infty((0,\infty))$ with $\varphi\equiv 1$ near $1$, 
$$
u|_{\partial\Omega}=S\varphi(|\xi'|_g)Tu|_{\partial\Omega}+O(e^{-C/h})_{L^2}.
$$
Next, 
$$
u|_{H}=\gamma_H(P+K_a)u|_{\partial\Omega}=\gamma_H(P+K_a)S\varphi(|\xi'|_g)Tu|_{\partial\Omega}+O(e^{-C/h})_{L^2}.
$$
Now, we have
\begin{align*}
&K_aS\varphi(|\xi'|_g)T\\
&=\frac{1}{(2\pi h)^{3n/2}}\int K_a(x',w)e^{\frac{i}{h}( \langle \exp_{z'}^{-1}(w'),\zeta'\rangle+\langle \exp_{y'}^{-1}(w'),\zeta'\rangle+\frac{i}{2}d(w',z')^2+\frac{i}{2}d(y',z'))}\varphi(|\zeta'|_{g(z')})dz'd\zeta'dw',
\end{align*}
Since $K_a$ is analytic, we may deforming the contour in $w'$ to $w'+i\delta\zeta'$, and use that $|\zeta'|>c>0$ to obtain
$$
K_aS\varphi(|\xi'|_g)T=O(e^{-C/h})_{L^2\to L^2},
$$
and hence
$$
u|_{H}=\gamma_HPS\varphi(|\xi'|_g)Tu|_{\partial\Omega}+O(e^{-C/h})_{L^2}.
$$

Now, note that $E_h$ naturally decomposes into a sum of operators acting on from $L^2(N_j)\to L^2(N_j)$ where $M=\sqcup N_j$ and $N_j$ are the connected components of $M$. We assume that $H_t=\varphi_{N_0}(N_0,t)$ with $\varphi_{N_0}$ as in~\eqref{e:phiDef} and write the component of $E_h$ acting on $N_0$ as $E_h^0$.
Therefore, with $E_h$ as in~\eqref{e:defInverse} and $\varphi\in C_c^\infty(0,1+\e)$ with $\varphi\equiv 1$ near $1$, we have by Lemma~\ref{l:itsAnInverse}
$$
E^0_h(u|_{H_t})  =  E^0_h \gamma_HPS\varphi(|\xi'|_g)Tu|_{\partial\Omega}   = Au|_{N_0}+E^0_hO(e^{-C/h})_{L^2}
$$
where $A$ is a semiclassical pseudodifferential operator on $\partial \Omega$ that is elliptic on 
$$
\{(x',\xi')\,:\, |\xi'|_{g(0,x')}=1\}.
$$
In particular, 
\begin{align*}
\|Au|_{N_0}\|&\leq \|E^0_h\|_{L^2\to L^2}(\|u|_{H_t}\|_{L^2(H_t)} +O(e^{-C/h}))\\
&\leq C_\e \sup_{x\in H_t,|\eta|_{g_{H_t}}\leq 1+\e}e ^{\Im\psi(x,\eta)/h}(\|u|_{H_t}\|_{L^2(H_t)}+O(e^{-C/h})_{L^2})
\end{align*}
Now, since 
$$
\operatorname{WF_h}(u|_{\partial\Omega})\subset S^*\partial\Omega,
$$
 and $A$ is elliptic on $S^*\partial \Omega,$ we have by e.g.~\cite[Theorem E.33]{DyZw:19}
$$
\|u\|_{L^2(N_0)}\leq C_\e\|Au\|_{L^2(N_0)}+O(h^\infty)\|u\|_{L^2(N_0)}.
$$
In particular, for $h$ small enough
$$
\|u\|_{L^2(N_0)}\leq C_\e\|Au\|_{L^2(N_0)}\leq C_\e \sup_{x\in H_t,|\eta'|_{g_{H_t}}\leq 1+\e}e ^{\Im \psi(x,\eta)/h}(\|u|_{H_t}\|_{L^2(H_t)}+O(e^{-C/h})_{L^2}).
$$
Therefore, for $t$ small enough, the proof is complete.
\end{proof}

\section{Carleman estimates under control assumptions: Proof of Theorem \ref{mainthm2}} \label{s:carleman}

Although the collar neighbourhood $U$ in which the Poisson representation in (\ref{poissonkernel}) is valid is fixed independent of $h$, the size of $U$ is difficult to make precise and could be quite small since it depends in a complicated fashion on the analyticity properies of $(\Omega,M,g)$. 
Our aim in the next section is to ``bootstrap" the lower bounds in Theorem \ref{mainthm1} further into the interior of $\Omega.$
To set notation, we let $N$ be a connected component of $\partial\Omega$ and $H=\varphi_N(N,t)$ with $0<t<r_{\max,N}$ and $\varphi_{N}$ as in~\eqref{e:phiDef}.
\begin{proof}
Let $(x',x_{n+1}): \Omega_{N} \to \R^{n} \times \R$ be the Fermi coordinates above adapted to the boundary component $N = \{ x_{n+1} = 0 \}$ and let $\Omega_{N}: = \{  0 \leq x_{n+1} < r_{\max,N} \}$ be the maximal  Fermi tube containing the hypersurface  $H_{\delta_0} = \{ x_{n+1} = \delta_0 \}$ with $\delta_0 < r_{\max,N}.$

\begin{rem} 
 We note here that  Fermi neighbourhood $\Omega_N$ of the boundary  depends only on the geometry of geodesic flow inside $\Omega$ and {\em not} on the analytic modulus of the data $(\partial \Omega, H, g).$ As such, it is often comparatively easy to determine the maximal tube width $\delta_0>0$ in Theorem \ref{mainthm1}. 
 \end{rem}

\subsection{Definition of cutoff functions}

 Let $\epsilon >0$ to be chosen small later such that 
 \begin{equation} \label{ep1}
 \epsilon < \min ( \frac{\ep_0}{4}, \delta_0, r_{\max,N}, - \delta_0),\end{equation}
 where we recall that $H_{\delta_0}= \{ x_{n+1} = \delta_0 \} \subset  \Omega_{N}(r_{\max,N})$ and $\ep_0 = \ep_0(M,\Omega,g)$ is the radius for which Theorem~\ref{mainthm1} holds.
 
   Let $\chi_{\epsilon}^+ \in C^{\infty}(\R;[0,1])$ with  
 
 $$ \text{supp} \, \partial \chi_{\epsilon}^+ , \subset \{ x \in \Omega_{N}; 0 < x_{n+1} < \epsilon \}$$ 
 and
 $$ \chi_{\epsilon}^+(x_{n+1}) = 1; \quad  x_{n+1} >  2 \epsilon ,\qquad\qquad \chi_{\epsilon}^+(x_{n+1}) = 0; \quad x_{n+1} <   \epsilon .$$ 
  
  We let $\chi_{\epsilon}^{-} \in C^{\infty}(\R;[0,1])$ be another cutoff be localized around the hypersurface $H = \{ x_{n+1} = \delta_0 \}$ with
 
  $$ \text{supp} \, \partial \chi_{\epsilon}^- , \subset \{  \delta_0 - \epsilon < x_{n+1} < \delta_0 + \epsilon \}$$ 
 and
 $$ \chi_{\epsilon}^-(x_{n+1}) = 1; \quad x_{n+1} < \delta_0 - 2 \epsilon ,\qquad\qquad \chi_{\epsilon}^-(x_{n+1}) = 0; \quad  x_{n+1} > \delta_0 + 2 \epsilon.$$ 
Finally, we set
\begin{equation} \label{cutoff}
\chi_{\epsilon}(x',x_{n+1}): = \chi_{\epsilon}^+(x_{n+1}) \cdot \chi_{\epsilon}^-(x_{n+1})   \in C^{\infty}_{0}(\R; [0,1]),\end{equation}\
where by Leibniz rule it follows that 
\begin{equation} \label{leib}
\text{supp} \,  \partial \chi_{\epsilon} \subset  \text{supp} \, \partial \chi_{\epsilon}^+ \cup \text{supp} \, \partial \chi_{\epsilon}^-  .
 \end{equation}\
Next, we set
$$ v_h:= \chi_{\epsilon} e^{\psi/h} u_h \in C^{\infty}_{0}(\Omega_{N}),$$
where $\psi \in C^{\infty}(\Omega_{N})$ is a weight function that is defined below. As usual, one then considers the conjugated operator $P_{\psi}(h):= e^{\psi/h} P(h) e^{-\psi/h} : C^{\infty}_{0}(\Omega_{N}) \to C^{\infty}_0(\Omega_{N})$ with principal symbol
$$p_{\psi}(x,\xi) = p(x, \xi + i \partial_{x} \psi ),$$
where $p(x,\xi) = |\xi|_g^2.$ In Fermi coordinates $(x',x_{n+1}): \Omega_{N} \to \R^{n+1},$
\begin{equation} \label{Fermi}
p(x,\xi) = \xi_n^2 + g(x,\xi'), \quad g(x,\xi') = g_{\partial}(x',\xi') + 2 x_{n+1} \kappa_{\partial}(x,\xi') \end{equation}
where $g$ is a quadratic form in $\xi'$. In the Taylor expansion (\ref{Fermi}), $g_{\partial}$ is the dual metric form on $T^*\partial \Omega$ induced from the interior and $\kappa_{\partial}(x',x_{n+1} =0,\xi')$ is the second fundamental form of the boundary.

\subsection{Carleman weight}

Fix $\delta>0$. We define the putative weight function to be 
\begin{equation} \label{weight}
\begin{gathered}
\psi_N(x_{n+1})=\int_0^{x_{n+1}}e^{\frac{1}{2}\int_0^sf_\delta (t) dt}ds \quad  x_{n+1} \in [0, r_{\max,N}],
\end{gathered}
\end{equation}
where $f_\delta\in C^\infty ([0,r_{\max,N}] )$ and satisfies
\begin{equation} \label{squeeze}
\begin{gathered}
\delta \leq f_\delta (t)- \sup_{ \{ (x',\xi'); g(t,x',\xi')=1 \} } \partial_{t} g(t,x',\xi')\, \leq 2\delta,
\end{gathered}
\end{equation}
so that 
\begin{equation}
\label{e:weightProp}
\begin{aligned}
( | \psi_N'(x_{n+1}) |^2 ) '&=f_\delta (x_{n+1}) \, (\psi_N'(x_{n+1}))^2\\
&\geq \Big( \sup_{ \{(x',\xi'); \,  g(x_{n+1},x',\xi') = 1 \}} \partial_{x_{n+1}}g(x_{n+1},x',\xi')\, \Big) \, (\psi_N'(x_{n+1}))^2 +\delta \, (\psi_N'(x_{n+1}))^2,\\
&=  \Big( \sup_{ \{(x',\xi');  \, g(x_{n+1},x',\xi') = |\psi_N '(x_{n+1})|^2 \}} \partial_{x_{n+1}}g(x_{n+1},x',\xi')\, \Big) \,  + \, \delta \, (\psi_N'(x_{n+1}))^2,\
\end{aligned}
\end{equation}
and
$$
\partial_{x_{n+1}}\psi_N|_{x_{n+1}=0}=1,\qquad  \partial_{x_n+1}^2\psi_N|_{x_{n+1}=0}=\tfrac{1}{2}f_\delta(0).
$$
The last line in (\ref{e:weightProp}) follows since $g(x,\xi')$ is quadratic in the fiber $\xi'$-variables.

%
%

To show that the function $\psi_N$ in (\ref{weight}) is a legitmate Carleman weight, we compute that  in Fermi coordinates $(x',x_{n+1}): \Omega_{N}(r_{\max,N}) \to \R^{n} \times \R,$
\begin{align*} p_{\psi_N}(x,\xi) &= ( \xi_{n+1} + i \partial_{x_{n+1}} \psi_N )^2 + g(x,\xi') \\
&=\xi_{n+1}^2+g(x,\xi')-(\partial_{x_{n+1}}\psi_N)^2+2i\xi_{n+1}\partial_{x_{n+1}}\psi_N.
\end{align*}

Since $\partial_{x_{n+1}}\psi_N\geq c>0$ it follows that
\begin{eqnarray} \label{char}
\text{Char}(p_{\psi_N})(x,\xi) =  \{ (x,\xi) \in T^* \Omega_{N};   g(x,\xi') =(\partial_{x_{n+1}}\psi_N)^2, \,\,\, \xi_{n+1} = 0 \} 
 \end{eqnarray}

Then, since $\Re p_{\psi_N} =  \xi_{n+1}^2 + g(x,\xi') - ( \partial_{x_{n+1}}\psi_N)^2$ and  $\Im p_{\psi_N} = 2 \partial_{x_{n+1}}\psi_N\xi_{n+1},$  a direct computation gives
\begin{equation} \label{comm}
\begin{aligned}
&\{ \Re p_{\psi_N}, \Im p_{\psi_N} \} (x,\xi)\\ &=  \big\{  \xi_{n+1}^2 + g(x,\xi') - ( \partial_{x_{n+1}}\psi_N)^2,  \,  2 \partial_{x_{n+1}}\psi_N\xi_{n+1} \big\}  \\
& = 2\partial_{x_{n+1}}\psi_N \cdot \Big( \partial_{x_{n+1}}[(\partial_{x_{n+1}}\psi_N)^2] - \partial_{x_{n+1}} g(x,\xi')  \Big) \end{aligned}\qquad(x,\xi) \in \text{Char}  \, p_{\psi_N}.
\end{equation}
 
 Then, since $\partial_{x_{n+1}}\psi_N\geq c \geq 0,$ it follows from~\eqref{e:weightProp} and~\eqref{comm} that
 
\begin{equation} \label{carlwt}
\{ \Re p_{\psi_N}, \Im p_{\psi_N} \} (x,\xi) \geq  C_1(\delta) | \partial_{x_{n+1}}  \psi_N |^2  \geq C_2(\delta) >0, \quad (x,\xi) \in \text{Char}(p_{\psi_N}). \end{equation} \\
Consequently $\psi_N$ is a legitimate Carleman weight in $\Omega_{N}(r_{\max,N})$ and so, by the subelliptic Carleman estimates (see e.g. \cite[Theorem 7.5]{Zw})
\begin{equation} \label{carleman}
\| P_{\psi_N}(h) v_h \|_{L^2(\Omega_{N} )}^2 \geq C h \, \| v_h \|_{H^1_h(\Omega_{N}) }^2. \end{equation}
\subsection{Lower bounds: completion of the proof}

Since $P(h) u_h = 0,$ it follows that $P_{\psi_N}(h) v_h = e^{\psi_N/h} [ P(h), \chi_{\epsilon}] u_h$. Since $[P(h), \chi_{\epsilon}]$ is an $h$-differential operator of order one supported in supp  $\partial \chi_{\epsilon},$  it follows from (\ref{carleman}) and (\ref{leib}), that with Carleman weight $\psi_N(x_{n+1})$ in (\ref{weight}),

\begin{eqnarray} \label{carleman2}
\| e^{\psi_N/h} [P(h),\chi_{\epsilon}] u_h \|^2_{\supp \partial \chi_{\epsilon}^{+}} + \| e^{\psi_N/h} [P(h),\chi_{\epsilon}] u_h \|^2_{\supp \partial \chi_{\epsilon}^{-}} 
  \geq C h \, \| e^{\psi_N/h} \chi_{\epsilon} u_h \|_{H_h^1(\Omega_{N})}^2. \end{eqnarray}\

%
%
%
%
  
 Since $\chi_{\epsilon}= 1$ on the set $\Gamma(\delta_0,\epsilon):= \{  2 \epsilon < x_{n+1} <  \delta_0 - 2\epsilon \},$ from (\ref{carleman2}),
 \begin{eqnarray} \label{carleman3}
\| e^{\psi_N /h} [P(h),\chi_{\epsilon}] u_h \|^2_{\supp \partial \chi_{\epsilon}^{+}} + \| e^{\psi_N /h} [P(h),\chi_{\epsilon}] u_h \|^2_{\supp \partial  \chi_{\epsilon}^{-}}  
  \geq C h \, \| e^{\psi_N /h}  u_h \|_{H^1_h( \Gamma(\delta_0,\epsilon) )}^2, \end{eqnarray}

 and so,
  \begin{eqnarray} \label{carleman4}
 h^2 \Big(  \| e^{\psi_N /h} u_h \|^2_{H^1_h( \supp \partial \chi_{\epsilon}^{+} ) } +  \| e^{\psi_N /h}  u_h \|^2_{ H^1_h( \supp \partial \chi_{\epsilon}^{-} ) }  \Big)  \geq C h  \| e^{\psi_N /h}  u_h \|_{H^1_h( \Gamma(\delta_0,\epsilon) )}^2. \end{eqnarray} \

To bound the first term on the LHS of (\ref{carleman4}) from above,  we recall  the upper bound from~\cite{GaTo:19}:
\begin{equation} \nonumber
|\partial_x^\alpha u_h(x)| \leq C_{\Omega,\alpha} \exp \,\big(  [  - x_{n+1} + C_{sup}(\Omega) x_{n+1}^2) ]/h\big){\|u\|_{L^2(N)}+Ce^{-1/Ch}}; \quad x \in \Omega_{N}(\ep_0).\end{equation}\
Since supp $\partial \chi_{\epsilon}^+ \subset \{ 0 < x_{n+1} < \epsilon   \}$  and $\epsilon < \epsilon_0,$ it follows that
\begin{align*} 
  h^2   \| e^{\psi_N /h} u_h \|^2_{H^1_h( \supp \partial \chi_{\epsilon}^{+} ) }  
 \leq  h^2C_{\Omega}^2\| e^{\psi_N /h} e^{[ -x_{n+1} + C_{sup}(\Omega) x_{n+1}^2)]/h} \|^2_{L^2 ( \{ 0 < x_{n+1} <  \epsilon \} )}\|u_h\|_{L^2(N)}^2+Ce^{-1/Ch}.\end{align*}
 
Using that $\psi_N(x_{n+1}) = x_{n+1} + \frac{1}{2}\psi_N'(0) x_{n+1}^2+O(x_{n+1}^3),$
$$\| e^{\psi_N /h} e^{[ -x_{n+1} + C_{sup}(\Omega) x_{n+1}^2)]/h} \|^2_{L^2 ( \{ 0 < x_{n+1} <  \epsilon \} )}\leq C_6 e^{  ( f_\delta(0) + C_{sup})  \epsilon^2 +C\e^3/ h},$$
and consequently, for the first term on the LHS of (\ref{carleman}),

\begin{equation} \label{small term}
 h^2   \| e^{\psi_N /h} u_h \|^2_{H^1_h( \supp \partial \chi_{\epsilon}^{+} ) } \leq C_{\Omega}^2 h^2 e^{  ( f_\delta(0) + C_{sup})  \epsilon^2 +C\e^3/ h}\|u_h\|_{L^2(N)}^2+Ce^{-1/Ch}.
  \end{equation} \

To estimate the RHS of (\ref{carleman4}) from below, we use the local $L^2$-restriction lower bounds in Theorem \ref{mainthm1} which gives that for  all $\tau \in [0,\ep_0] $ and any $\delta_1 >0,$ then, with $h \in (0,h_0(\delta_1,\e_0)],$
$$ \| u_h \|_{L^2( \{ x_{n+1} = \tau \} )} \geq C(\delta_1) e^{ -  [ \, (\tau + \frac{1}{4}f_0(0)\tau^2 +C\tau^3)(1+\delta_1)\, ] / h}\|u_h\|_{L^2(N)}-Ce^{-1/Ch}. $$\

Now, let $\e<\e_1\ll \min(\delta,\e_0)$. since $\{ \frac{\ep_1}{2} < x_{n+1} < \ep_1 \} \subset \Gamma(\delta_0,\ep),$ for the RHS in (\ref{carleman4}),
\begin{align*}
&\| e^{\psi_N /h}  u_h \|_{H^1_h( \Gamma(\delta_0,\epsilon) )}^2\\
& \geq  \| e^{\psi_N /h}  u_h \|_{L^2( \{ \frac{\ep_1}{2} < x_{n+1} < \ep_1\} }^2 \\
&\geq C(\delta_1) \int_{\tau = \frac{\ep_1}{2} }^{\tau =  \ep_1 } e^{ 2 \tau  + \frac{1}{2}f_\delta(0) \tau^2-C\tau^3  /h} e^{- (2 \tau  + \tfrac{1}{2}f_0(0) \tau^2 +C\tau^3)(1+\delta_1) ) /h} \, dt\|u_h\|^2_{L^2(N)}-Ce^{-1/Ch} \\
& = C(\delta_1)  \int_{\tau = \frac{\ep_1}{2} }^{\tau =  \ep_1 } e ^{\delta \tau^2 -C\tau^3-\delta_1(2\tau+\frac{1}{2}f_0(0)\tau^2) /h} \, d\tau \|u_h\|^2_{L^2(N)}-Ce^{-1/Ch} \\
&\geq C(\delta_1) \int_{\tau = \frac{\ep_1}{2} }^{\tau =  \ep_1 } e ^{  \frac{\delta}{2} \tau^2 /h} \, d\tau\|u_h\|^2_{L^2(N)}-Ce^{-1/Ch} ,\end{align*}
where the last line follows by choosing $\e_1\ll \delta$ and $\delta_1\ll \e_1\delta$ and noting that from (\ref{squeeze}), 
$f_{\delta}(0) - f_0(0) \geq \delta.$

Consequently, the end result is the following lower bound for the RHS in (\ref{carleman4}):
\begin{equation} \label{RHS}
\| e^{\psi_N /h}  u_h \|_{H^1_h( \Gamma(\delta_0,\epsilon) )}^2
\geq  C(\delta_1)  \e_1 e^{ \frac{\delta\e_1^2}{2h}}\|u_h\|^2_{L^2(N)}-e^{-1/Ch}. \end{equation}\
%
%
%
By possibly shrinking $\ep >0$ further so that $\e\ll \sqrt{\delta}\e_1$, it follows from the upper bound in (\ref{small term}) and the lower bound in (\ref{RHS}), that the first term on the LHS of (\ref{carleman4}) can  be absorbed in the RHS. Consequently,
  \begin{eqnarray} \label{carleman5}
  \| e^{\psi_N /h}  u_h \|^2_{ H^1_h( \supp \partial \chi_{\epsilon}^{-} ) }    \geq C(\ep)\|u_h\|^2_{L^2(N)}-Ce^{-1/Ch}, \end{eqnarray}
 and  that, together with elliptic regularity, completes the proof of Theorem \ref{mainthm2}.

\end{proof}

\section{Eigenfunction goodness estimates for Cauchy data: Proof of Theorem \ref{mainthm3}} \label{upper}



\begin{proof}  
Let $0\leq t<r_{\max,\partial\Omega}$ and $H_t:=\varphi_{\partial\Omega}(\partial\Omega,t)$ so that in Fermi coordinates $(x',x_{n+1})$ around $\partial \Omega$, 
$$
H_t=\{(x',t)\}.
$$
Let $U_{H_t} \subset \mathring{\Omega}$ be the domain interior to $\Omega$ bounded by $H_t$ and for any fixed $\ep>0,$  let $U_{H_t,\ep} \Subset U_{H_t}$ be a compact manifold with boundary, $H_{t,\ep},$ strictly contained in $U_{H_t}$ with  $\frac{\ep}{2} \leq d(H_{t,\ep}, H_t) \leq \max_{(x,y) \in H_{t,\ep} \times H_t}  d(x,y)  \leq \ep.$
Then, by e.g.~\cite[Chapter 4]{Au:82}, there exists a Green's function $G \in {\mathcal D}'(U_{H_t}  \times U_{H_t})$ satisfying
$$ -\Delta_x G(x,y) = \delta_x(y), \quad (x,y) \in U_{H_t} \times  U_{H_t},$$
with $G(\cdot, \cdot) \in C^{\infty}( (U_{H_t} \times  U_{H_t}) \setminus \{ x=y \} ).$  

Then, for $x \in U_{H_t,\ep} $ an application of Green's formula gives
\begin{equation} \label{green}
h u_h(x) = \int_{H_t} G(x,s) h \partial_{\nu} u_h(s) d\sigma(s)  - h \int_{H_t} \partial_{\nu(s)} G(x,s) u_h(s) d\sigma(s). \end{equation} 

Since $d(x,H_t) > \ep$ when $x \in U_{H_t,\ep}$ and so, $G \in C^{\infty}(U_{H_t,\ep}, H_t),$ differentiation of (\ref{green}) in the $x$-variables gives a similar formula for the derivatives $\partial_{x_k} u_h.$
An application of Cauchy-Schwarz then implies that

\begin{equation} \label{pl2}
h^2 \| u_h \|_{ H^1_h (U_{H_t,\ep}) }^2 \leq C(\ep) \Big(  \| u_h \|_{L^2(H_{t})}^2 + \| h \partial_{\nu} u_h \|_{L^2(H_t)}^2 \Big).
\end{equation} \

Finally, by Theorem \ref{mainthm2}, applied with $H_{t,\ep}=\{d(x,\partial\Omega)=t+\e\}$, for any $\e>0$, 
$$
\|u\|_{H^1_h(U_{H_t,\ep})}\geq C_\e e^{(-\psi_{\partial\Omega}(t)-\e)/h},
$$
which completes the proof.

\end{proof}

%
%
%
%
%
%
%
%
%

\bibliography{biblio}
\bibliographystyle{alpha}

\end{document}